\documentclass[a4paper,11pt]{amsart}

\usepackage{mathpazo}
\usepackage{amssymb}
\usepackage[pagebackref,
    ,pdfborder={0 0 0}
    ,urlcolor=black,a4paper,hypertexnames=false]{hyperref}
\hypersetup{pdfauthor={Daniel Fauser, Stefan, Friedl, Clara L\"oh},pdftitle=Integral approximation of simplicial volume of graph manifolds}

\usepackage{pgf}
\usepackage{tikz}
\usetikzlibrary{decorations.pathmorphing}
\usepackage[all]{xy}
\SelectTips{cm}{}

\newtheorem{thm}{Theorem}[section]
\newtheorem{prop}[thm]{Proposition}

\newtheorem{cor}[thm]{Corollary}

\newtheorem{question}[thm]{Question}

\theoremstyle{remark}
\newtheorem{rem}[thm]{Remark}
\newtheorem{example}[thm]{Example}

\theoremstyle{definition}
\newtheorem{defi}[thm]{Definition}

\usepackage{amsfonts}
\newcommand{\Z}{\mathbb{Z}}
\newcommand{\Q}{\mathbb{Q}}
\newcommand{\R}{\mathbb{R}}
\newcommand{\N}{\mathbb{N}}

\DeclareMathOperator{\rk}{rk}
\DeclareMathOperator{\tors}{tors}

\def\epsilon{\varepsilon}
\DeclareMathOperator{\res}{res}
\DeclareMathOperator{\Aut}{Aut}

\DeclareMathOperator{\ModZsn}{{\sf Mod}_\Z^{\sf sn}}
\DeclareMathOperator{\SBP}{{\sf SBP}}

\def\linfz#1{%
  L^\infty(#1;\Z)}

\def\fcl#1{%
  [#1]}

\def\args{\;\cdot\;}

\def\sv#1{\|#1\|}
\def\isv#1{\|#1\|_\Z}

\makeatletter
\newcommand\norm{\bBigg@{0.8}}

\makeatother
\newcommand{\ifsv}[2][norm]{\csname #1l\endcsname\bracevert\!#2\!%
                            \csname #1r\endcsname\bracevert}
\newcommand{\stisv}[2][norm]{\indnorml[#1]{#2}{\Z}{\infty}}
\newcommand{\indnorml}[4][flex]{\csname #1l\endcsname\|#2%
                                 \csname #1r\endcsname\|_{#3}^{#4}\mathclose{}}

\def\actson{\curvearrowright}

\def\ucov#1{%
  \widetilde{#1}
}

\usepackage{color}
\usepackage{pdfcolmk}

\def\draftinfo{}

\author{Daniel Fauser}
\author{Stefan Friedl}
\author{Clara L\"oh}
\title[Integral approximation of simplicial volume of graph manifolds]
      {Integral approximation of simplicial volume\\ of graph manifolds}
\date{\today.\ \copyright{\ D.~Fauser, S.~Friedl, C.~L\"oh 2017}. 
    This work was supported by the CRC~1085 \emph{Higher Invariants} 
    (Universit\"at Regensburg, funded by the DFG).
    \draftinfo\\
     MSC~2010 classification: 55N10, 57N65, 57M27}
\begin{document}

\begin{abstract}
  Graph manifolds are manifolds that decompose along tori into pieces
  with a tame $S^1$-structure.  In this paper, we prove that the
  simplicial volume of graph manifolds (which is known to be zero) can be
  approximated by integral simplicial volumes of their finite
  coverings. This gives a uniform proof of the vanishing of rank
  gradients, Betti number gradients and torsion homology gradients for
  graph manifolds.
\end{abstract}

\maketitle


\section{Introduction}

Many classical invariants from topology and group theory admit 
meaningful gradient invariants, which are defined by a
stabilisation and normalisation process over finite coverings and finite
index subgroups, respectively.  For example, Betti number gradients
coincide in many cases with the corresponding
$L^2$-invariants~\cite{lueckapprox}. 

We will consider an approximation problem for simplicial
volume: The simplicial volume~$\sv M$ of an oriented closed connected
(topological) $n$-manifold is the infimum of the $\ell^1$-norms of fundamental
cycles with $\R$-coefficients of~$M$. A related gradient
invariant is the stable integral simplicial volume~$\stisv M$ of~$M$,
defined as the infimum of the normalised integral simplicial volumes
of finite coverings of~$M$ (Section~\ref{sec:simvol}). 

\begin{question}[integral approximation problem for simplicial volume]
  For which oriented closed connected manifolds~$M$ do we have
  \[ \sv M = \stisv M \;?
  \]
\end{question}

In the present paper, we prove that the simplicial volume of graph
manifolds satisfies integral approximation. We introduce a notion of
graph manifolds as manifolds that decompose along tori into pieces
that admit a tame $S^1$-structure (Section~\ref{sec:graphmfds}); our
definition of graph manifolds excludes the case of spherical
$3$-mani\-folds, but it does include all other classical graph
manifolds in dimension~$3$ as well as higher-dimensional examples.

\begin{thm}\label{thm:main}
  Let $M$ be an oriented closed connected graph manifold (in the sense
  of Definition~\ref{def:graphmfd}) with residually finite fundamental
  group. Then
  \[ \sv M = \stisv M = 0.
  \]
  More generally: Let $(\Gamma_j)_{j \in \N}$ be a descending chain of finite
  index subgroups of~$\pi_1(M)$ with trivial intersection and let $(M_j)_{j \in \N}$
  be the corresponding tower of finite coverings. Then
  \[ \sv M = \lim_{j \rightarrow \infty} \frac{\isv{M_j}}{[\pi_1(M) : \Gamma_j]} = 0.
  \]
\end{thm}

Vanishing of the ordinary simplicial volume follows already 
from work of Gromov~\cite{vbc}, Yano~\cite{yano}, and Soma~\cite{soma}.

In addition, the following classes of manifolds are known to
satisfy integral approximation for simplicial volume: closed surfaces of genus at
least~$1$~\cite{vbc}, closed hyperbolic
$3$-manifolds~\cite{flps}, closed aspherical manifolds with residually
finite amenable fundamental group~\cite{flps}, compact manifolds with
``non-trivial'' $S^1$-action~\cite{fauser}, as well as certain glueings along
tori~\cite{fauserloeh}. In contrast, approximation fails uniformly for
higher-dimensional hyperbolic manifolds~\cite{ffm} and it fails for closed
manifolds with non-abelian free fundamental group~\cite[Remark~3.9]{flps}.

\begin{question}
  Do all oriented closed connected $3$-manifolds~$M$ with infinite
  fundamental group satisfy the approximation identity~$\sv M = \stisv
  M$ for simplicial volume?
\end{question}

\subsection*{Applications}

Vanishing of stable integral simplicial volume, in particular,
provides a uniform proof of the vanishing of the following gradient
invariants:

\begin{cor}[gradient invariants of graph manifolds]\label{cor:gradient}
  Let $M$ be an oriented closed connected graph manifold with
  residually finite fundamental group. Then the following hold:
  \begin{enumerate}
  \item The rank gradient of~$\pi_1(M)$ is~$0$. 
  \item If $R$ is a principal ideal domain and~$k \in \N$, then
    the $\rk_R H_k(\args;R)$-gradients of~$M$ are~$0$.
  \item If $k\in\N$, then the $\log \;\mathopen| \tors H_k(\args;\Z)|$-gradients  
    of~$M$ are~$0$.
  \item The Euler characteristic of~$M$ is~$0$.
  \end{enumerate}
\end{cor}
\begin{proof}
  By Theorem~\ref{thm:main}, we have~$\stisv M = 0$.  The rank
  gradient estimate hence follows from the fact that stable integral
  simplicial volume is an upper bound for the rank
  gradient~\cite{loehrg}.  Alternatively, one can also apply the proof
  strategy for Theorem~\ref{thm:main} via
  Theorem~\ref{thm:mainergodic} to derive the triviality of the rank
  gradient from the corresponding results on
  cost~\cite[Theorem~1]{abertnikolov}\cite[Chapter~29--37]{kechrismiller}.

  For the homology and torsion homology gradients, we consider a
  descending chain~$(\Gamma_j)_{j \in \N}$ of finite index subgroups
  of~$\pi_1(M)$ with trivial intersection~$\bigcap_{j \in \N}
  \Gamma_j$; because $\pi_1(M)$ is assumed to be residually finite,
  such chains do exist. Let $(M_j)_{j \in \N}$ be the corresponding
  tower of covering spaces of~$M$.  Then $\rk_R H_k(M_j;R) \leq
  \isv {M_j}$ for all~$j \in \N$~\cite[Lemma~4.1]{flps} and hence
  (Theorem~\ref{thm:main})
  \[ \limsup_{j \rightarrow \infty} \frac{\rk_R H_k(M_j;R)}{[\pi_1(M):\Gamma_j]}
  \leq \lim_{j \rightarrow \infty} \frac{\isv {M_j}}{[\pi_1(M) : \Gamma_j]}
  = 0.
  \]
  Moreover, we have~\cite[proof of Theorem~1.6]{flps}
  \[ \limsup_{j \rightarrow \infty} \frac{\log\;\mathopen|\tors H_k(M_j;R)|}{[\pi_1(M):\Gamma_j]}
  \leq
  \log (n+1) \cdot {{n+1} \choose {k+1}} \cdot 
  \lim_{j \rightarrow \infty} \frac{\isv {M_j}}{[\pi_1(M) : \Gamma_j]}   
  \]
  (where $n := \dim M$) and the right-hand side is zero by Theorem~\ref{thm:main}.
    
  The Euler characteristic is the alternating sum of the rational
  Betti numbers and it is multiplicative under finite coverings;
  therefore, the proof of vanishing of the $\dim_\Q H_*(\args;\Q)$-gradients
  shows that $\chi(M) = 0$.
\end{proof}

In dimension~$3$, these consequences can also be proved by a direct
calculation~\cite[Theorem~3.14 and Corollary~3.18]{Meumertzheim}.
Furthermore, for $3$-manifolds, Corollary~\ref{cor:gradient}~(3)
is also a consequence of recent work by L\^{e}~\cite{Le}.

\begin{rem}[$L^2$-Betti numbers of graph manifolds]
  Let $M$ be an oriented closed connected graph manifold with
  residually finite fundamental group and $k \in \N$. Then the
  triviality of the $\dim_\Q H_k(\args;\Q)$-gradient of~$M$
  (Corollary~\ref{cor:gradient}) and L\"uck's approximation
  theorem~\cite{lueckapprox} imply that
  \[ b_k^{(2)}(M) = 0.
  \]
  Conversely, one can also prove the vanishing of the $L^2$-Betti
  numbers via $L^2$-methods (proceeding along the lines of our
  inductive proof of Theorem~\ref{thm:mainergodic}) and then deduce
  the vanishing of the $\dim_\Q H_k(\args;\Q)$-gradients via L\"uck's
  approximation theorem.
\end{rem}

These results provide evidence for an affirmative answer to Gromov's
question~\cite[p.~232]{gromovasymptotic} whether the vanishing of
simplicial volume of aspherical closed manifolds implies the vanishing
of their $L^2$-Betti numbers/Euler characteristic. 

\subsection*{Strategy of proof}

It is tempting to try to prove Theorem~\ref{thm:main} by constructing
good finite coverings by hand, taking advantage of the
$S^1$-structure.  However, the bookkeeping for the glueing steps would
be quite tricky. It is much more efficient to pass to a more general
setting with more flexible coefficient modules (see
Section~\ref{sec:simvol} for notation and terminology) that has better
inheritance properties and allows for a straightforward induction
proof. In this extended setting, we prove the following vanishing
result:

\begin{thm}\label{thm:mainergodic}
  Let $M$ be an oriented compact connected graph manifold with
  fundamental group~$\Gamma$, and let $\alpha = \Gamma \actson X$ be
  an essentially free standard $\Gamma$-space.  Then
  \[ \ifsv{M,\partial M} ^\alpha = 0.
  \]
\end{thm}

Theorem~\ref{thm:mainergodic} can be proved by induction over the
graph structure; the base case of a manifold with tame $S^1$-structure
can be solved using methods by Fauser~\cite{fauser}, the induction
step requires a glueing argument similar to previous results by Fauser
and L\"oh~\cite{fauserloeh}. The main contribution of the present
paper is to give a proper formalisation of this induction argument and
adapting the glueing argument to the case of multiple boundary
components; this will be done in the setting of fundamental groupoids
and local coefficients.

Theorem~\ref{thm:main} is then a consequence of the fact that taking
the profinite completion of the fundamental group as coefficient action
leads to the stable integral simplicial volume.

\subsection*{Organisation of this article}

We first explain the setup of generalised graph manifolds
(Section~\ref{sec:graphmfds}) and of the simplicial volumes needed for
the proof of the main theorem (Section~\ref{sec:simvol}). The proof of
Theorem~\ref{thm:main} and Theorem~\ref{thm:mainergodic} is given in
Section~\ref{sec:proofs}.

\section{Graph manifolds}\label{sec:graphmfds}

Graph manifolds are (topological) manifolds that can be decomposed along tori into
pieces that admit a tame $S^1$-structure. In order to keep track of
such decompositions/glueings and facilitate induction proofs, we will
use a more formal description via graphs.

\subsection{Tame $S^1$-structures}

We will first introduce the basic building blocks of graph manifolds, namely 
manifolds that admit a tame $S^1$-structure.

Recall that a subspace~$Y$ of a topological space~$X$ is
\emph{$\pi_1$-injective} if for every basepoint~$y \in Y$ the
map~$\pi_1(Y,y) \longrightarrow \pi_1(X,x)$ induced by the inclusion
is injective.

\begin{defi}[tame $S^1$-structure]\label{def:tame}
  A compact manifold~$M$ of dimension~$n \in \N$ admits a \emph{tame $S^1$-structure}
  if there exists an~$m \in \N$ and pairwise disjoint compact $n$-dimensional
  submanifolds~$N_1, \dots, N_m$ of~$M^\circ$ with the following properties:
  \begin{itemize}
  \item The complement~$M' := M \setminus \bigcup_{j=1}^m N_j^\circ$ admits a smooth structure
    and a smooth $S^1$-fibre bundle structure~$M' \longrightarrow B$ 
    over a compact smooth $(n-1)$-manifold~$B$ (possibly non-oriented and possibly with boundary)
    with $\pi_1$-injective (in~$M$) fibres.
  \item For each~$j \in \{1,\dots,m\}$ the manifold~$N_j$ is homotopy equivalent to
    a torus of dimension at most~$n-2$ and $N_j$ is a $\pi_1$-injective subspace of~$M$. 
  \end{itemize}
\end{defi}

\begin{example}\label{exa:asphericaltame}
  If $M$ is an aspherical oriented closed connected smooth manifold
  that has a smooth free $S^1$-action, then $M$ admits a tame
  $S^1$-structure in the sense of Definition~\ref{def:tame}: By the
  slice theorem, the canonical projection~$M \longrightarrow M/S^1$
  and the given $S^1$-action form an $S^1$-principal bundle over the
  smooth base manifold~$M/S^1$~\cite[Theorem~15.3.4]{tomdieck}. Moreover, in this
  situation, the corresponding fibres (i.e., the orbits of the action)
  are $\pi_1$-injective~\cite[Corollary~1.43]{lueckl2}.
\end{example}

\begin{example}[Seifert manifolds]\label{exa:seiferttame}
  Every compact Seifert $3$-manifold that is \emph{not} finitely
  covered by~$S^3$ admits a tame $S^1$-structure. Indeed, for the
  $N_j$ we take the tubular neighbourhoods around the singular fibres
  provided by the definition of a Seifert $3$-manifold. The regular
  fibres are $\pi_1$-injective~\cite[Lemma~3.2]{scott} and thus the
  first condition above is satisfied. Furthermore, by definition of a
  Seifert $3$-manifold each $N_j$ contains a regular fibre $F$ such
  that the inclusion induced map $\Z\cong \pi_1(F)\longrightarrow \pi_1(N_j) \cong
  \Z$ is a monomorphism. It follows that the $N_j$ are also
  $\pi_1$-injective in $M$, thus the second condition above is
  satisfied.
\end{example}

\subsection{Graphs}

We formalise the glueing of multiple manifolds with multiple boundary
components via graphs.  In the following, we will use unoriented
graphs, possibly with multi-edges and loops. More precisely, we model
graphs by their vertices, edges, and the map that assigns the incident
vertices to the edges:

\begin{defi}[graph]\label{def:graph}
  A \emph{graph} is a triple~$X = (V,E,\partial)$,
  where $V$ and $E$ are sets (the sets of \emph{vertices}
  and \emph{edges} of~$X$, respectively) and $\partial$
  is a map of the type
  \[ \partial \colon E
     \longrightarrow \bigl\{ \{v,w\} \bigm| v,w \in V \bigr\}.
  \]
\end{defi}

\subsection{Graphs of manifolds}

For the purpose of this paper, we will use the following terminology
and conventions on graphs of manifolds.  A graph of manifolds is a
graph, where each vertex is decorated with a manifold with a tame
$S^1$-structure and $\pi_1$-injective torus boundary components and
where each edge describes a glueing of two of these boundary
components.  For~$m \in \N$, we write~$T^m := (S^1)^m$ for the
standard $m$-torus. More precisely:

\begin{defi}[graph of manifolds]
  Let $n \in \N$.  A \emph{graph of $n$-dimensional manifolds} is a
  triple~$\Gamma = (X, (M_v)_{v \in V}, (f_e)_{e \in E})$
  consisting of the following components:
  \begin{itemize}
  \item A graph~$X = (V,E,\partial)$ in the sense of Definition~\ref{def:graph}
    with finite sets~$V$ and~$E$.
  \item A family~$(M_v)_{v\in V}$ of oriented connected $n$-manifolds that
    admit a tame $S^1$-structure (Definition~\ref{def:tame}) and whose
    boundary components are $\pi_1$-injective $(n-1)$-tori.
  \item A family~$(f_e)_{e \in E}$ of maps with the following properties:
    \begin{itemize}
    \item
      If $e \in E$ with $\partial e =\{v,w\}$ and $v \neq w$, then $f_e$ is a map
      (a homeomorphism onto its image)
      \[ f_e \colon T^{n-1} \times \{v,w\} \longrightarrow \partial M_v \sqcup \partial M_w
      \]
      that induces an orientation-reversing homeomorphism between the corresponding
      boundary components of~$M_v$ and~$M_w$. 
    \item
      If $e \in E$ is a loop (i.e., $\partial e =\{v\}$ for some~$v\in V$),
      then $f_e$ is a map
      \[ f_e \colon T^{n-1} \times \{0,1\} \longrightarrow \partial M_v
      \]
      that maps the two tori to different boundary components of~$M_v$ and 
      induces an orientation-reversing homeomorphism between these boundary
      components.
    \item If $e,e' \in E$ satisfy~$e \neq e'$, then the
      images of~$f_e$ and $f_{e'}$ are disjoint. 
    \end{itemize}
  \end{itemize}
\end{defi}

Given a graph of manifolds, we can glue the vertex manifolds as specified
by the edge maps:

\begin{defi}[geometric realisation]
  Let $\Gamma = (X, (M_v)_{v\in V}, (f_e)_{e \in E})$ be
  a graph of $n$-manifolds with $X = (V,E,\partial)$.
  Then the \emph{geometric realisation of~$\Gamma$}
  is the oriented $n$-manifold
  \[ M(\Gamma) := \biggl(\coprod_{v \in V} M_v\biggr) \bigm/ \sim,
  \]
  where $\sim$ is the identification induced by the maps~$(f_e)_{e \in E}$. 
\end{defi}

Alternatively, one can describe the geometric realisation as follows: Let 
$\Gamma = (X = (V,E,\partial), (M_v)_{v\in V}, (f_e)_{e\in E})$ be a graph
of $n$-manifolds. Then the barycentric subdivision~$X'$
of~$X$ is a canonically oriented graph.
The graph~$\Gamma$ of manifolds gives rise to an $X'$-shaped diagram
of manifolds by associating the torus with each of the new vertices
obtained by barycentric subdivision and associating the (components of) 
the glueing maps with the edges. Then $M(\Gamma)$ is nothing but the colimit
of this diagram.

Applying the functor~$\pi_1$ (with a suitable treatment of basepoints)
to a graph~$\Gamma$ of manifolds gives rise to a graph of groups. Then
the fundamental group of the geometric realisation~$M(\Gamma)$ is
isomorphic to the fundamental group of this graph of
groups.

\subsection{Graph manifolds}

\begin{defi}[graph manifold]\label{def:graphmfd}
  Let $n \in \N$. A \emph{graph manifold} of dimension~$n$ is
  a manifold homeomorphic to the geometric realisation of a
  graph of $n$-dimensional manifolds; recall that all vertex
  manifolds of such a graph admit a tame $S^1$-structure.
\end{defi}

By construction, graph manifolds are orientable (through the 
orientation that is compatible with the vertex manifolds of an
appropriate graph of manifolds), compact,
and all boundary components are tori. Moreover, a graph manifold is
connected if and only if the underlying graph is connected. 

\begin{example}\label{exa:graphmfds}
  This definition subsumes the following classes of graph manifolds:
  \begin{itemize}
  \item Graph manifolds in the sense of $3$-manifolds, i.e.,
    orientable prime $3$-manifolds with empty or toroidal boundary
    such that all JSJ-com\-ponents are Seifert manifolds~\cite{AFW},
    except for $3$-manifolds that are finitely covered by~$S^3$
    (because Seifert JSJ-components are admissible vertex manifolds in
    our setting of graph manifolds,
    Example~\ref{exa:seiferttame}). Note that these manifolds all have
    residually finite fundamental group~\cite{hempel}.
  \item Higher-dimensional graph manifolds in the sense
    of Frigerio, Lafont, Sisto~\cite{frigeriolafontsisto}
    (because products of manifolds with~$S^1$ are admissible
    vertex manifolds).
  \end{itemize}
\end{example}

\section{Simplicial volumes}\label{sec:simvol}

Simplicial volumes of a manifold count the minimal number of singular simplices
that are needed to represent the fundamental class, weighted by a norm on
the coefficients. 

\subsection{Simplicial volume}

Classically, simplicial volume is defined with respect to constant
coefficients~\cite{vbc,mapsimvol,frigeriobc}:

\begin{defi}[simplicial volume]
  Let $n \in \N$ and let $M$ be an oriented compact connected $n$-manifold.
  Then the \emph{simplicial volume} and the \emph{integral simplicial volume}
  of~$(M,\partial M)$ are defined by
  \begin{align*}
    \sv {M,\partial M}
    := \inf \bigl\{ |c|_1
    \bigm| \; & \text{$c \in C_n(M;\R)$ is a relative $\R$-fundamental cycle}\\
              & \text{of~$(M,\partial M)$}
                \bigr\}, 
    \\
    \isv {M, \partial M} 
    := \inf \bigl\{ |c|_1
    \bigm| \; & \text{$c \in C_n(M;\Z)$ is a relative $\Z$-fundamental cycle}\\
              & \text{of~$(M,\partial M)$}
              \bigr\}
    .
  \end{align*}
  Here, $|\cdot|_1$ denotes the $\ell^1$-norm on~$C_n(M;\R)$ and $C_n(M;\Z)$, respectively,
  associated with the basis of singular simplices.
  
  The \emph{stable integral simplicial volume} of~$(M,\partial)$ is defined as
  \begin{align*}
    \stisv {M,\partial M} := \inf \Bigl\{ \frac {\isv {N,\partial N}} d \Bigm|
    \; & d \in \N_{>0} \text{ and}\\
       & \text{$N$ is a $d$-sheeted covering of~$M$}\Bigr\}.
  \end{align*}
\end{defi}

In the context of $L^2$-invariants or approximation questions for
simplicial volume, we have to pass to twisted/local coefficients
(Section~\ref{subsec:normedlocal}).

\subsection{Normed local coefficients}\label{subsec:normedlocal}

The definition of integral foliated simplicial volume involves twisted
coefficients~\cite{gromovmetric,mschmidt,loehpagliantini}. However, in the context
of manifolds with multiple boundary components, it is more convenient
to work in the framework of local coefficients than in the framework
of twisted coefficients. Therefore, we will briefly recall local
coefficients~\cite{steenrod} and their relation with twisted coefficients.

\begin{defi}[normed local coefficient system]
  Let $M$ be a topological space; we denote the fundamental
  groupoid of~$M$ by~$\pi(M)$~\cite[Chapter~2.5]{tomdieck}.
  A \emph{normed local coefficient system on~$M$}
  is a functor
  \[ \pi(M)  \longrightarrow \ModZsn,
  \]
  where $\ModZsn$ is the category of all semi-normed $\Z$-modules
  (and norm-non-increasing homomorphisms).
\end{defi}

\begin{defi}[homology with local coefficients]
  Let $M$ be a topological space and let $L \colon \pi(M) \longrightarrow \ModZsn$
  be a normed local coefficient system on~$M$. The \emph{chain complex~$C_*(M;L)$
  of~$M$ with local coefficients in~$L$} is defined as follows:
  For~$n \in \N$, we set
  \[ C_n(M;L) := \bigoplus_{x \in M} \bigoplus_{\sigma \in S_n(M,x)} L(x) \cdot \sigma
  \]
  where $S_n(M,x)$ is the set of all singular $n$-simplices~$\sigma \colon \Delta^n
  \longrightarrow M$ in~$M$ with~$\sigma(e_0) = x$. Moreover, we define the
  boundary operator
  \begin{align*}
    \partial \colon C_n(M;L) \longrightarrow C_{n-1}(M;L)
  \end{align*}
  by $\Z$-linear extension of
  \[ \partial( a \cdot \sigma)
     := \bigl(L(\sigma[0,1])\bigr)(a) \cdot \sigma[0]
     + \sum_{j=1}^n (-1)^j \cdot a \cdot \sigma[j]
  \]
  for all~$x \in M$, $a\in L(x)$, $\sigma \in S_{n}(M,x)$; here, $\sigma[j]$
  denotes the $j$-th face of~$\sigma$ and $\sigma[0,1]$ denotes the composition
  of~$\sigma$ with the canonical parametrisation~$[0,1] \longrightarrow \Delta^n$
  of the $0-1$-edge of~$\Delta^n$. Then $\partial \circ \partial = 0$.
  
	We define the \emph{$\ell^1$-semi-norm on~$C_*(M;L)$}
	as follows:
	For all~$n\in \N$ and every chain $\sum_{i=1}^k a_j\cdot\sigma_j \in C_n(M;L)$
	in \emph{reduced form},
	i.e., the $\sigma_j$'s are pairwise distinct, let
	\[ \Bigl|\sum_{i=1}^k a_j \cdot \sigma_j \Bigr|_1 :=
	      \sum_{i=1}^k |a_j|_{L(\sigma_j(e_0))},
	\]
	where~$|\cdot|_{L(\sigma_j(e_0))}$ denotes the given semi-norm on the 
	$\Z$-module~$L\bigl(\sigma_j(e_0)\bigr)$.
  
  The homology of~$C_*(M;L)$ is called \emph{homology of~$M$ with local
    coefficients in~$L$} and denoted by~$H_*(M;L)$.
  The semi-norm on~$C_*(M;L)$ induces the \emph{$\ell^1$-semi-norm on~$H_*(M;L)$} via
	\begin{align*}
		\|\alpha\|_1 := \inf\bigl\{ |c|_1 \bigm|~&c\in C_n(M;L),\ 
	        \partial c= 0,\ [c]= \alpha \in H_n(M;L) \bigr\}
	\end{align*}
	for all~$n\in \N$ and all~$\alpha\in H_n(M;L)$.
	
	Let~$U$ be a subspace of~$M$. Let~$I\colon \pi(U)\longrightarrow \pi(M)$
	be the induced functor of the 
	inclusion~$U\subset M$ and let~$L_U := L\circ I$.
	Then~$C_*(U;L_U)$ is a subcomplex of~$C_*(M;L)$ and we define the
	\emph{chain complex of M relative to~$U$ with local coefficients in~$L$} by
	\[ C_*(M,U;L) := C_*(M;L)/C_*(U;L_U);
	\]
	the boundary operator~$\partial$ on~$C_*(M;L)$ induces a well-defined
	boundary operator on~$C_*(M,U;L)$ that we denote by~$\partial$ again.
	Then we define the \emph{homology of~$M$ relative to~$U$ with local
	coefficients in~$L$} by
	\[ H_*(M,U;L) := H_*\bigl(C_*(M,U;L)\bigr)
	\]
	and we obtain the \emph{$\ell^1$-semi-norm on~$H_*(M,U;L)$} as follows:
	For all~$n\in \N$ and all~$\alpha\in H_n(M,U;L)$ let
	\begin{align*}
		||\alpha||_1 := \inf\bigl\{ |c|_1 \bigm|~&c\in C_n(M;L),  
	      \partial c\in C_{n-1}(U;L_U), [c]= \alpha \in H_n(M,U;L) \bigr\}.
	\end{align*}
\end{defi}

\begin{rem}[local vs.\ twisted coefficients]
  Let $M$ be a path-connected topological space and let $x_0 \in M$;
  then $\Aut_{\pi(M)} x_0$ is nothing but the fundamental group~$\pi_1(M,x_0)$
  based at~$x_0$. 
  \begin{itemize}
  \item If $L \colon \pi(M) \longrightarrow \ModZsn$ is a normed
    local coefficient system on~$M$, then $L(x_0)$ has
    the structure of a normed right-$\pi_1(M,x_0)$-module.
  \item
    Conversely, if $A$ is a normed right-$\pi_1(M,x_0)$-module, then we can
    construct a local coefficient system~$L_A$ on~$M$ as follows: For
    all~$x \in M \setminus\{x_0\}$, we choose a path~$\gamma_x \colon [0,1]
    \longrightarrow M$ from~$x_0$ to~$x$ and we let~$\gamma_{x_0}$
    be the constant path at~$x_0$.
    \begin{itemize}
    \item For~$x \in M$, we set~$L_A(x) := A$.
    \item For~$[\gamma \colon x \rightarrow y]_* \in \pi(M)$, we set
      \begin{align*}
        L([\gamma]_*) \colon 
        L_A(x) = A& \longrightarrow A = L_A(y)
        \\
        a & \longmapsto a \cdot [\gamma_{x} * \gamma * \overline \gamma_y]_*.
      \end{align*}
    \end{itemize}
  \end{itemize}
  It is easy to verify that $L_{L(x_0)} \cong L$ and $L_A(x_0) = A$.
\end{rem}

\begin{prop}[homology with local vs.\ twisted coefficients]\label{prop:homlocal}
  Let $M$ be a path-connected topological space that admits
  a universal covering~$\pi_M \colon \ucov M \longrightarrow M$, let $x_0 \in M$, and let
  $L \colon \pi(M) \longrightarrow \ModZsn$ be a normed
  local coefficient system on~$M$. Moreover, let $n \in \N$,
  let $D \subset \ucov M$ be a set-theoretic fundamental
  domain for the deck transformation action of~$\pi_1(M,x_0)$
  on~$\ucov M$, and let $(\gamma_x)_{x\in M}$ be a family of
  paths connecting~$x_0$ with the other points in~$M$ (such that $\gamma_{x_0}$
  is constant).
  Then
  \begin{align*}
    C_n(M;L) & \longrightarrow L(x_0) \otimes_{\Z \pi_1(M,x_0)} C_n(\ucov M;\Z)
    \\
    a \cdot \sigma & \longmapsto L([\gamma_{\sigma(e_0)}]^{-1}_*)(a) \otimes \widetilde \sigma 
    \\
    L(x_0) \otimes_{\Z \pi_1(M,x_0)} C_n(\ucov M;\Z) & \longrightarrow C_n(M;L)
    \\
    a \otimes \sigma & \longmapsto L([\gamma_{\pi_M \circ \sigma(e_0)}]_*)(a) \cdot \pi_M \circ \sigma
    \\
    \text{with~$\sigma(e_0) \in D$}&
  \end{align*}
  are mutually inverse isomorphisms that do not increase the norm.
  Here, $\widetilde \sigma$ denotes the unique $\pi_M$-lift of~$\sigma$ to~$\ucov M$
  with~$\widetilde \sigma(e_0) \in D$. These maps yield chain maps
  and hence induce mutually inverse isomorphisms
  \[ H_k(M;L) \cong H_k\bigl(M; L(x_0)\bigr)
  \]
  that are isometric with respect to the induced semi-norms on homology.
\end{prop}
\begin{proof}
  This is a straightforward calculation.
\end{proof}

\subsection{Relative integral foliated simplicial volume}

Integral foliated simplicial volume is a simplicial volume that
uses dynamical systems of the fundamental groupoid as local coefficients.
For basic notions on (actions) on standard Borel probability spaces
we refer to the literature~\cite{kechris,kechrismiller}.

\begin{defi}[parameter spaces]
  Let $G$ be a groupoid. A \emph{standard $G$-space} is functor~$G
  \longrightarrow \SBP$ into the category of standard Borel
  probability spaces (with probability measure preserving
  transformations). A standard $G$-space is \emph{essentially free [ergodic]},
  if for every point in~$G$ the induced group action is essentially free [ergodic].
  Recall that a group action is \emph{essentially free} almost all points
  have trivial isotropy.
\end{defi}

\begin{defi}[normed local coefficients associated with standard actions]
  Let $G$ be a groupoid and let $\alpha \colon G \longrightarrow \SBP$ be
  a standard $G$-space. We define the associated normed local
  coefficient system~$\linfz \alpha \colon G \longrightarrow \ModZsn$
  by
  \[ L(x) := \linfz {\alpha(x)}
  \]
  (with the $L^1$-``norm'') for all points~$x$ of~$G$
  and
  \begin{align*}
    L(g) \colon \linfz {\alpha(x)} & \longrightarrow \linfz {\alpha(y)}
    \\
    f &\longmapsto f \circ \alpha(g^{-1})
  \end{align*}
  for all morphisms~$g \colon x \longrightarrow y$ of~$G$.
\end{defi}

\begin{defi}[parametrised relative fundamental class]
  Let $n \in \N$,  
  let $M$ be an oriented compact $n$-manifold (possibly with boundary),
  and let $\alpha$ be a standard $\pi(M)$-space.
  Then the image
  \[ \fcl {M,\partial M}^\alpha \in H_n\bigl(M,\partial M; \linfz \alpha\bigr)
  \]
  of the integral fundamental class~$\fcl {M,\partial M}_\Z \in H_n(M,\partial M;\Z)$
  under the change of coefficients map induced by the inclusion of (the constant
  system)~$\Z$ into~$\linfz \alpha$ (as constant functions) is the
  \emph{$\alpha$-parametrised (relative) fundamental class of~$(M,\partial M)$}. 
\end{defi}

\begin{prop}\label{prop:boundaryrel}
  Let $n \in \N$,  
  let $M$ be an oriented compact $n$-manifold (possibly with boundary),
  and let $\alpha$ be a standard $\pi(M)$-space.
  Then
  \[ \partial \bigl(\fcl{M,\partial M}^\alpha\bigr) = \fcl{\partial M}^{\res^{\pi(M)}_{\pi(\partial M)}\alpha}
     \in H_{n-1}(\partial M; \res^{\pi(M)}_{\pi(\partial M)} \alpha).
  \]
  Here, $\res^{\pi(M)}_{\pi(\partial M)} \alpha$ is the composition
  of~$\alpha \colon \pi(M) \longrightarrow \SBP$ with the groupoid
  morphism~$\pi(\partial M) \longrightarrow \pi(M)$ induced by the
  inclusion~$\partial M \longrightarrow M$. 
\end{prop}
\begin{proof}
  We only need to check this equality for integral (relative)
  fundamental classes, where it is well-known. 
\end{proof}

\begin{defi}[relative integral foliated simplicial volume]
  Let $n \in \N$,  
  let $M$ be an oriented compact $n$-manifold (possibly with boundary),
  and let $\alpha$ be a standard $\pi(M)$-space. Then the \emph{$\alpha$-parametrised
    simplicial volume of~$(M,\partial M)$} is defined by
  \[ \ifsv {M,\partial M}^\alpha
  := \inf \bigl\{ |c|_1
          \bigm| \text{$c \in C_n(M;\alpha)$ represents~$\fcl{M,\partial M}^\alpha$}
          \bigr\};
  \]
  recall that if $\alpha = \pi(M) \actson (X,\mu)$ and $c =
  \sum_{j=1}^k f_j \cdot \sigma_j \in C_n(M;\alpha)$ is in reduced
  form, then
  \[ |c|_1 = \sum_{j=1}^k \int_X |f_j| \;d\mu.
  \]
  The \emph{relative integral foliated simplicial volume~$\ifsv {M,\partial M}$
    of~$(M,\partial M)$} is the infimum over all parametrised simplicial volumes
  of~$(M,\partial M)$ (the isomorphism types of standard $\pi(M)$-spaces form a
  set). 
\end{defi}

This definition is compatible with the usual definition of
parametrised and integral foliated simplicial volume in terms of
twisted coefficients~\cite{mschmidt,loehpagliantini}. For simplicity,
we only formulate this in the closed case:

\begin{prop}[comparison with the twisted definition]
  Let $n \in \N$, let $M$ be an oriented closed connected $n$-manifold,
  let $x_0 \in M$, and let $\alpha$ be a standard $\pi(M)$-space. Then
  $\ifsv M ^\alpha
  $
  coincides with the $\ell^1$-semi-norm of the parametrised
  fundamental class in homology~$H_n(M; \linfz{\alpha(x_0)})$ with twisted coefficients
  in the $\Z\pi_1(M,x_0)$-module~$\linfz{\alpha(x_0)}$.
\end{prop}
\begin{proof}
  This is a special case of Proposition~\ref{prop:homlocal}.
\end{proof}

Furthermore, the previous proposition also extends to the case
of manifolds with boundary.
So, in principle, one could always get away with the twisted version. However,
working with twisted coefficients requires the choice of a basepoint. When
dealing with manifolds with multiple boundary components, this leads to
an unpleasant overhead. 

\subsection{A local criterion}

Given a top-dimensional parametrised homology class, we will need a
local criterion that decides whether this class coincides with the
parametrised fundamental class or not. As a first step, we briefly
recall parametrised Poincar\'e-Lefschetz duality (which, in particular,
allows to compute the top-dimensional parametrised homology).

\begin{prop}[parametrised Poincar\'e-Lefschetz duality]\label{prop:PD}
  Let $n \in \N$,  
  let $M$ be an oriented compact connected $n$-manifold (possibly with boundary),
  and let $\alpha$ be a standard $\pi(M)$-space. Then the relative cap-product
  induces isomorphisms
  \[ \args \cap [M]_\Z \colon H^{n-k}\bigl(M;\linfz \alpha\bigr)
     \longrightarrow H_k\bigl(M,\partial M; \linfz \alpha\bigr)
  \]
  for all~$k \in \N$.
\end{prop}
\begin{proof}
  This is a special case of Poincar\'e-Lefschetz duality with twisted
  coefficients: If $M$ is triangulable, the pair~$(M,\partial M)$ is a
  connected Poincar\'e pair in the sense of
  Wall~\cite[Theorem~2.1]{wall} and connected Poincar\'e pairs satisfy
  Poincar\'e-Lefschetz duality with twist\-ed
  coefficients~\cite[Lemma~1.2]{wallpc}.

  For the general case of topological manifolds one can, for example,
  adapt the classical proof~\cite[Chapter~VI]{bredon} to the setting
  of twisted coefficients.
\end{proof}

\begin{prop}[local criterion]\label{prop:localcrit}
  Let $n \in \N$, let $M$ be an oriented compact $n$-manifold
  (possibly with boundary), let $\alpha$ be a standard
  $\pi(M)$-space, let $U \subset M^\circ$ be a non-empty compact connected 
  $n$-dimensional submanifold (possibly with boundary), and 
  let $c \in C_n(M;\linfz \alpha)$ be a relative cycle of~$(M,\partial M)$.
  Then the following are equivalent:
  \begin{enumerate}
  \item The chain~$c$ is an $\alpha$-parametrised relative fundamental
    cycle of~$(M,\partial M)$.
  \item In~$H_n\big(M, M\setminus U^\circ; \linfz \alpha\bigr)$, the chain~$c$
    represents the class $\varphi [U, \partial U]_{\alpha'}$, 
    where
    $\alpha'$ is the restriction of~$\alpha$ to~$\pi(U)$
    and 
    \[ \varphi \colon H_n\bigl(U,\partial U;\linfz {\alpha'}\bigr)
    \longrightarrow H_n\bigl(M,M\setminus U^\circ;\linfz \alpha \bigr) 
    \]
    is induced by the canonical
    transformation~$\linfz {\alpha'} \longrightarrow \linfz \alpha$.
  \end{enumerate}
\end{prop}
\begin{proof}
  The inclusions~$(M,\partial M) \longrightarrow (M, M\setminus U^\circ)$, 
  $(U,\partial U) \longrightarrow (M,M\setminus U^\circ)$ give rise to the left
  hand side of the commutative diagram in Figure~\ref{fig:localcrit}.

  \begin{figure}
    \begin{center}
      \makebox[0pt]{%
    $ \xymatrix{%
      H_n(U,\partial U;\Z)
      \ar[r]
      \ar[d]_-{\cong}
      &
      H_n(U,\partial U;\linfz{\alpha'})
      \ar[d]^-{\varphi}
      \ar[r]
      &
      H_n(U,U\setminus V^\circ; \linfz{\alpha'}_{\pi(U)})
      \ar[d]
      \\
      H_n(M,M\setminus U^\circ;\Z)
      \ar[r]
      &
      H_n(M, M\setminus U^\circ;\linfz \alpha)
      \ar[r]
      &
      H_n(M, M\setminus V^\circ; \linfz{\alpha}_{\pi(M)})
      \\
      H_n(M,\partial M;\Z)
      \ar[r]
      \ar[u]^-{\cong}
      &
      H_n(M,\partial M;\linfz \alpha)
      \ar[u]
      \ar[r]
      &
      H_n(M, M\setminus V^\circ; \linfz{\alpha}_{\pi(M)})
      \ar@{=}[u]
      }
        $}
      \end{center}

    \caption{Proving the local criterion for parametrised fundamental classes}
    \label{fig:localcrit}
  \end{figure}

  For the right hand side, we proceed as follows: Let $V \subset U^\circ \subset M$
  be an embedded closed $n$-ball. The local coefficient systems $\linfz{\alpha}_{\pi(M)}$
  etc.\ are the coinvariants of the original systems~$\linfz \alpha$, i.e., they
  are obtained by dividing out the action of the automorphisms (that is the fundamental groups)
  at each point. Hence, by construction, these local coefficient systems can be
  viewed as constant coefficients and thus lead to ordinary homology groups.

  If $c$ satisfies the first condition, then $c$ represents~$\varphi [U,\partial U]_{\alpha'}$ 
  in the relative group~$H_n(M,M\setminus U^\circ;\linfz \alpha)$ because the isomorphisms of the
  leftmost column in Figure~\ref{fig:localcrit} are compatible with the
  corresponding integral fundamental classes.

  Conversely, let $c$ represent $\varphi [U,\partial U]_{\alpha'}$
  in~$H_n(M,M\setminus U^\circ;\linfz \alpha)$. Then $c$ represents
  the image of the ordinary fundamental class~$[U,U\setminus
    V^\circ]_\Z$, and thus of~$[M,M\setminus V^\circ]_\Z$, in~$H_n(M,M
  \setminus V^\circ;\linfz \alpha_{\pi(M)})$. Therefore, $c$ satisfies
  the hypothesis of the local criterion for embedded
  balls~\cite[Proposition~3.9]{fauser}; it should be noted that the
  version of the cited local criterion can also for local coefficients
  be derived from parametrised Poin\-ca\-r\'e-Lef\-schetz duality
  (Proposition~\ref{prop:PD}) in the same way as for twisted coefficients.
  Applying this local criterion to the
  ball~$V \subset M$ then implies that $c$ is an $\alpha$-parametrised
  fundamental cycle of~$(M,\partial M)$.
\end{proof}

\section{Simplicial volumes of graph manifolds}\label{sec:proofs}

We will first prove Theorem~\ref{thm:mainergodic} and then we will
derive Theorem~\ref{thm:main}. In order to prove Theorem~\ref{thm:mainergodic}, 
one can either perform all glueings at once or only glue along one torus at
a time (and then proceed by induction). 
We prefer the latter version. When combining chains along tori, we need
a way to fill boundaries efficiently: 

\begin{prop}[parametrised UBC for tori]\label{prop:UBClocal}
  Let $n \in \N_{>0}$, let $G := \pi(T^n)$, and let $\alpha$ be an
  essentially free standard $G$-space. Then $C_*(T^n;\alpha)$ satisfies
  the uniform bounday condition (UBC) in every degree, i.e.:
  For every~$k \in \N$ there is a constant~$K \in \R_{>0}$ such that 
  for every null-homologous cycle~$c \in C_k(T^n;\alpha)$ there
  exists a chain~$b \in C_{k+1}(T^n;\alpha)$ with
  \[ \partial b = c
     \quad\text{and}\quad
     |b|_1 \leq K \cdot |c|_1.
  \]
\end{prop}
\begin{proof}
  By the correspondence between local and twisted coefficients on
  the chain level (Proposition~\ref{prop:homlocal}), this is a direct
  consequence of the parametrised uniform boundary condition for tori
  formulated in terms of twisted
  coefficients~\cite[Theorem~1.3]{fauserloeh}.
\end{proof}

\subsection{Vertex manifolds}

We first treat the base case of $S^1$-bundles; in a second step, we then
use a glueing argument and UBC to treat the case of general tame $S^1$-structures. 

\begin{prop}\label{prop:bundlecase}
  Let $M$ be an oriented compact connected smooth $n$-manifold that is the total space
  of a smooth $S^1$-bundle~$p\colon M\longrightarrow B$ over a compact smooth
  $(n-1)$-manifold~$(B,\partial B)$. 
  Then
  \[ \ifsv {M,\partial M}^\alpha = 0
  \]
  holds for all standard $\pi(M)$-spaces~$\alpha$ whose restrictions to all 
  fibres are essentially free.
\end{prop}
\begin{proof}
	We choose a triangulation of~$B$ that is fine enough such that the bundle~$p$ is trivial
	over every simplex in this triangulation. As in Yano's proof for vanishing of simplicial volume
	of oriented closed connected smooth manifolds with non-trivial smooth $S^1$-action~\cite{yano}, 
	we define a sequence
	\[ M_{n-1} \stackrel{p_{n-2}} \longrightarrow M_{n-2} \longrightarrow \dots
			\longrightarrow M_1 \stackrel{p_0} \longrightarrow M_0 := M
	\]
	of hollowings of~$M$. Let~$X_0 := p^{-1}(B^{(0)})$ be the pre-image of the
	$0$-skeleton of~$B$. We define~$p_0$ to be the hollowing 
	at~$X_0\subset M_0$~\cite[Section~2]{yano}, i.e., we obtain $M_1$ from~$M_0$ by 
	removing a (small) tubular neighbourhood of~$X_0$.
	Now, we inductively define for all~$j\in \{1,\dots, n-1\}$ the map~$p_j$ to be the
	hollowing at~$X_j\subset M_j$, where~$X_j$ is the pullback of the $j$-skeleton of~$B$ along
	$p\circ p_0 \circ p_1 \circ \dots \circ p_{j-1}$.
	Let~$B^{[n-1]}$ denote the set of $(n-1)$-simplices in the
	chosen triangulation of~$B$.
	Furthermore, for every~$\Delta \in B^{[n-1]}$	let
	\[ \Delta_{n-1} \longrightarrow \Delta_{n-2} \longrightarrow \dots
				\longrightarrow \Delta_1 \longrightarrow \Delta_0 = \Delta
	\]
	be the induced sequence of restricted hollowings at the skeleta of~$\Delta$. 
	Then
	\[ M_{n-1} \cong \coprod_{\Delta\in B^{[n-1]}} \Delta_{n-1} \times S^1
	\]
        and $M_{n-1}$ inherits an orientation from~$M$. 
	
	Let~$\alpha$ be a standard $\pi(M)$-space whose restrictions to the fibres are
        essentially free. We set~$\alpha_0 := \alpha$. 
	For every~$j\in \{1,\dots, n-1\}$ let~$\alpha_j$ be the standard $\pi(M_j)$-space that
	is given by restricting~$\alpha$ along~$p_0 \circ p_1\circ \dots \circ p_{j-1}$. Furthermore,
	we denote by~$P_j$ the induced map of~$p_j$ from the $\alpha_j$-parametrised chain complex
	of~$M_j$ to the $\alpha_{j-1}$-parametrised chain complex of~$M_{j-1}$.
	
	For every simplex~$\Delta \in B^{[n-1]}$, we choose an integral relative fundamental
	cycle~$z_\Delta\in C_{n-1}(\Delta_{n-1};\Z)$ that is compatible with the CW-structure
	on~$\partial \Delta$ given by the	sequence of hollowings above.
	By hypothesis, the restriction of~$\alpha_{n-1}$
	to each~$\Delta \times S^1$ yields an essentially free standard
        $\pi(S^1)$-space~$\alpha_\Delta$.
	Let~$\epsilon~\in \R_{>0}$. Then, for every~$\Delta\in B^{[n-1]}$ there exists an
	$\alpha_\Delta$-parametrised fundamental cycle~$c^{S^1}_\Delta$ of~$S^1$
        with
	\[ |c_\Delta^{S^1}|_1 < \epsilon
	\]
	such that~$z_\Delta \times c^{S^1}_\Delta$ is a
        $\alpha_\Delta$-parametrised relative fundamental cycle
        of~$\Delta \times S^1$ (with the orientation inherited
        from~$M$); Schmidt~\cite[Proposition~5.30]{mschmidt} stated this for
        ergodic parameter spaces, but his proof also works for essentially free
        parameter spaces.
        We set
	\[ z := \sum_{\Delta \in B^{[n-1]}} z_\Delta \times c^{S^1}_\Delta \in C_n(M_{n-1};\alpha_{n-1}).
	\]
	Let~$A:= \max \bigl\{ |z_\Delta|_1 \bigm| \Delta\in B^{[n-1]}\bigr\}$.
	Then we have
	\[ |z|_1 \leq n \cdot |B^{[n-1]}| \cdot A \cdot \epsilon.
	\]
        Therefore, $z$ is a parametrised relative fundamental cycle of~$M_{n-1}$ with
        small norm. Starting with~$z$, we wish to construct a parametrised
        relative fundamental cycle of~$M$ of small norm. 
	
	Indeed, we are now in the same situation as in the proof of the
        analogous vanishing result for parametrised simplicial volumes
        of smooth manifolds with non-trivial smooth
        $S^1$-actions~\cite[Theorem~1.1, Remark~6.4]{fauser}, the only
        difference being that we had to do one more step in the
        sequence of hollowings to obtain a trivial $S^1$-bundle. We
        then proceed as in the cited proof, adapting the chain
	\[ P_0\circ P_1\circ \dots \circ P_{n-2} (z) \in C_n(M;\alpha)
	\]
	to get an $\alpha$-parametrised relative fundamental cycle
        of~$M$ without increasing the norm too much: We
        inductively investigate the defect of the pushforward of~$z$
        in~$M_j$ from being a parametrised relative fundamental cycle
        and we fill the defect with the help of
        Proposition~\ref{prop:UBClocal}.
\end{proof}

\begin{prop}[vertex manifolds]\label{prop:simvoltame}
  Let $M$ be an oriented connected compact manifold that admits a tame $S^1$-structure.
  Then
  \[ \ifsv {M,\partial M}^\alpha = 0
  \]
  holds for all essentially free standard $\pi(M)$-spaces~$\alpha$. 
\end{prop}
\begin{proof}
  Let $n := \dim M$. 
  Because $M$ admits a tame $S^1$-structure, there exists an~$m \in \N$ and
  pairwise disjoint compact submanifolds~$N_1, \dots, N_m$
  (with boundary) of dimension~$n$ with the following properties (Definition~\ref{def:tame}):
  \begin{itemize}
  \item The complement~$M' := M \setminus \bigcup_{j=1}^m N_j^\circ$
    admits a smooth $S^1$-bundle structure. 
  \item For each~$j \in \{1,\dots,m\}$ the manifold~$N_j$ is homotopy equivalent to
    a torus of dimension at most~$n-2$ and $N_j$ is a $\pi_1$-injective subspace of~$M$. 
  \end{itemize}
  Let $\alpha$ be an essentially free standard $\pi(M)$-space and let $\alpha'$ be
  the induced standard $\pi(M')$-space. Because the fibres are $\pi_1$-injective,
  the restriction of~$\alpha'$ to each fibre is essentially free. Therefore, $M'$
  and $\alpha'$ satisfy the hypotheses of Proposition~\ref{prop:bundlecase}.
  
  Let $\varepsilon \in \R_{>0}$. By Proposition~\ref{prop:bundlecase}, there
  exists a chain~$z' \in C_n(M';\alpha')$ representing~$\fcl{M',\partial M'}_{\alpha'}$
  with
  \[ |z'|_1 \leq \varepsilon.
  \]
  
  For~$j \in \{1,\dots,m\}$ let $z_j := (\partial z')|_{N_j} \in
  C_{n-1}(N_j;\alpha_j)$ be the $N_j$-component of~$\partial
  z'$. Here, $\alpha_j$ denotes the restriction of~$\alpha$
  to~$N_j \subset M' \subset M$; because $\alpha$ is
  essentially free and $N_j$ is a $\pi_1$-injective subspace of~$M$,
  also $\alpha_j$ is an essentially free standard $\pi(N_j)$-space. 
  As $N_j$ is homotopy equivalent to a torus of dimension at most~$n-2$,
  the chain~$z_j$ is null-homologous in~$C_*(N_j;\alpha_j)$ for dimension
  reasons. Hence, we can apply Proposition~\ref{prop:UBClocal} (and
  homotopy invariance of UBC~\cite[Proposition~3.15]{fauser}) to obtain
  a chain~$b_j \in C_n(N_j;\alpha_j)$ with
  \[ \partial b_j = z_j
     \quad\text{and}\quad
     |b_j|_1 \leq K_j \cdot |z_j|_1 
  \]
  (where $K_j$ is a UBC-constant for~$C_{n-1}(N_j; \alpha_j)$). 
  We now set  
  \[ z := z' - \sum_{j=1}^m b_j \in C_n(M;\alpha)
  \]
  (using the obvious inclusions between the parametrised
  chain complexes). By construction, we have 
  \[ |z|_1 \leq \varepsilon + \sum_{j=1}^m K_j \cdot (n+1) \cdot \varepsilon
     \quad\text{and}\quad
     \partial z = \partial z'.
  \]
  The local criterion (Proposition~\ref{prop:localcrit}) shows that $z$
  represents~$\fcl {M,\partial M}_\alpha$ (because $z$ restricts
  to the relative fundamental cycle~$z'$ of~$(M',\partial M')$).
  Therefore,
  \[ \ifsv {M,\partial M} ^\alpha \leq |z|_1 \leq \varepsilon + \sum_{j=1}^m K_j \cdot (n+1) \cdot \varepsilon.
  \]
  Taking $\varepsilon \rightarrow 0$ proves~$\ifsv {M,\partial M}^\alpha =0$.
\end{proof}

For non-spherical Seifert $3$-manifolds it was already known that the
stable integral simplicial volume
is~$0$~\cite{loehpagliantini}. However, for the induction step we will
need a more general vanishing result than just for stable integral
simplicial volume. Therefore, already for the treatment of graph manifolds
in dimension~$3$ the dynamical version of simplicial volume is helpful.

\subsection{Edge glueings}

\begin{prop}[glueings along tori]\label{prop:glueing}
  Let $n \in \N_{\geq 2}$ and let $(M_1,\partial M_1)$ and $(M_2,\partial
  M_2)$ be oriented compact connected $n$-manifolds with boundary.
  Let $T_1 \subset \partial M_1$ and $T_2 \subset \partial M_2$
  be $\pi_1$-injective components of~$\partial M_1$ and $\partial M_2$, respectively,
  that are homeomorphic to the torus~$T^{n-1}$.
  Let $f \colon T_1 \longrightarrow T_2$ be an orientation-reversing
  homeomorphism, let
  \[ M := M_1 \cup_f M_2
  \]
  be the oriented compact connected $n$-manifold obtained by glueing~$M_1$
  and~$M_2$ along the boundary components~$T_1$ and~$T_2$ via~$f$,
  let $G := \pi(M_1) *_{\pi(f)} \pi(M_2)$ be the corresponding pushout
  groupoid, and let $\alpha$ be an essentially
  free standard $G$-space with
  \[ \ifsv{M_1,\partial M_1}^{\res^G_{\pi(M_1)} \alpha} = 0
     \quad\text{and}\quad
     \ifsv{M_2,\partial M_2}^{\res^G_{\pi(M_2)} \alpha} = 0.
  \]
  Then~$\ifsv {M,\partial M}^\alpha=0$. In particular, $\ifsv {M,\partial M} = 0$.
\end{prop}
\begin{proof}
  We proceed as in the case with a single boundary
  component~\cite[Proposition~10.3]{fauserloeh}:   
  In order to keep the notation simple, we view $M_1$ and $M_2$ as
  subspaces of~$M$ and identify~$T_1$ and~$T_2$ via~$f$.

  By the Seifert and van Kampen theorem for fundamental groupoids, the
  inclusions of~$M_1$ and $M_2$ into~$M$ induce an isomorphism~$G
  \cong \pi(M)$. Moreover, as the boundary components are
  $\pi_1$-injective, we also know that the canonical maps~$\pi(M_1)
  \longrightarrow \pi(M)$ and $\pi(M_2) \longrightarrow \pi(M)$ are
  injective at every base-point. Therefore, the restrictions $\alpha_1 :=
  \res^G_{\pi(M_1)} \alpha$ and $\alpha_2:= \res^G_{\pi(M_2)} \alpha$
  of the essentially free $G$-space~$\alpha$ are essentially free;
  hence, also $\alpha_0 := \res^G_{\pi(T_1)}\alpha = \res^G_{\pi(T_2)} \alpha$ is
  essentially free.

  Let $K \in \R_{>0}$ be an $(n-1)$-UBC constant for~$C_*(T_1;\alpha_0)$
  (Proposition~\ref{prop:UBClocal}). Let $\varepsilon \in \R_{>0}$.
  Because of~$\ifsv{M_1,\partial M_1}^{\alpha_1} = 0$ and $\ifsv{M_2,\partial M_2}^{\alpha_2} = 0$
  there exist parametrised relative fundamental cycles~$c_1 \in C_n(M_1;\alpha_1)$
  as well as $c_2 \in C_n(M_2;\alpha_2)$ with
  \[ |c_1|_1 \leq \varepsilon
     \quad\text{and}\quad
     |c_2|_1 \leq \varepsilon.
  \]
  Then
  \[ c_0 := (\partial c_1) |_{T_1} + (\partial c_2) |_{T_2} \in C_{n-1}(T_1;\alpha_0)
  \]
  is a null-homologous cycle in~$C_{n-1}(T_1;\alpha_0)$ (because the glueing
  map~$f$ is orientation-reversing and $(\partial c_1)|_{T_1}$ and $(\partial c_2)|_{T_2}$ are
  $\alpha_0$-parametrised fundamental cycles of~$T_1$ by Proposition~\ref{prop:boundaryrel});
  by construction, 
  \[ |c_0|_1  \leq 2 \cdot (n+1) \cdot \varepsilon.
  \]
  By the uniform boundary condition on the torus~$T_1 = T_2$, there exists
  a chain~$b \in C_n(T_1;\alpha_0)$ with
  \[ \partial b = c_0
     \quad\text{and}\quad
     |b|_1 \leq K \cdot |c_0|_1.
  \]
  Then
  $c:= c_1 + c_2 - b \in C_n(M;\alpha)
  $ 
  is a cycle that
  satisfies
  \[ |c|_1 \leq 2\cdot \varepsilon + K \cdot 2 \cdot (n+1) \cdot \varepsilon.
  \]
  Moreover, $c$ is an $\alpha$-parametrised relative fundamental cycle
  of~$(M,\partial M)$ (Proposition~\ref{prop:localcrit}).
  
  Taking the infimum over all~$\varepsilon \in \R_{>0}$ shows that
  $\ifsv {M,\partial M}^\alpha = 0$.
\end{proof}

\begin{prop}[self-glueing along tori]\label{prop:selfglueing}
  Let $n \in \N_{\geq 2}$ and let $(M,\partial M)$ be an oriented
  compact connected $n$-manifold with boundary. Let $T_1,T_2 \subset \partial M$
  be two different $\pi_1$-injective components of~$\partial M$ that
  are homeomorphic to the torus~$T^{n-1}$. Let $f \colon T_1 \longrightarrow T_2$
  be an orientation-reversing homeomorphism, let
  \[ N := M / (T_1 \sim_f T_2)
  \]
  be the oriented compact connected $n$-manifold obtained by glueing~$M$
  to itself along~$T_1, T_2$ via~$f$, let $G := \pi(M) *_{\pi(f)}$ be the
  corresponding HNN-extension groupoid, and let $\alpha$ be an essentially free standard
  $G$-space with
  \[ \ifsv{M,\partial M}^{\res^G_{\pi(M)}\alpha} = 0.
  \]
  Then $\ifsv {N,\partial N}^\alpha = 0$. In particular, $\ifsv {N,\partial N} = 0$.
\end{prop}
\begin{proof}
  As in the proof of Proposition~\ref{prop:glueing}, one can take 
  a small parametrised relative fundamental cycle of~$(M,\partial M)$,
  and then use the uniform boundary condition on~$T_1 \cong T_2$ to construct
  a small parametrised relative fundamental cycle of~$(N,\partial N)$.  
\end{proof}

\subsection{Proof of Theorem~\ref{thm:mainergodic}}

\begin{proof}[Proof of Theorem~\ref{thm:mainergodic}]
  If $\Gamma$ is a graph of manifolds, then instead of performing
  all glueings in the geometric realisation~$M(\Gamma)$ at once,
  we can also do them step by step, glueing one pair of boundary tori
  after another. Therefore, we can prove Theorem~\ref{thm:mainergodic}
  by induction over the number of edges of~$\Gamma$.

  The base case of this induction is a graph of manifolds without
  edges, i.e., a disjoint union of vertex manifolds; this case is handled
  in Proposition~\ref{prop:simvoltame}.

  In the induction step, we have to distinguish two cases:
  \begin{itemize}
  \item In case of a glueing corresponding to an edge connecting two
    different connected components of the underlying graph, we apply
    Proposition~\ref{prop:glueing}.
  \item In case of a glueing corresponding to an edge connecting
    vertices in the same connected component of the underlying graph
    (this includes the case of loops), we apply
    Proposition~\ref{prop:selfglueing}.
    \qedhere
  \end{itemize}  
\end{proof}

\subsection{Proof of Theorem~\ref{thm:main}}

\begin{proof}[Proof of Theorem~\ref{thm:main}]
  Let $\alpha$ be the standard $\pi_1(M)$-space given by the profinite
  completion of the residually finite group~$\pi_1(M)$~\cite[Section~2.1]{flps}.
  Then \cite[Theorem~6.6, Remark~6.7]{loehpagliantini}
  \[ \stisv M = \ifsv M ^\alpha.
  \]
  On the other hand, $\ifsv M ^\alpha = 0$, by
  Theorem~\ref{thm:mainergodic} (if $\pi_1(M)$ is residually finite,
  the action on the profinite completion is free). Therefore, $\stisv
  M = 0$. Because of the
  sandwich~\cite[Proposition~6.1]{loehpagliantini}
  \[ 0 \leq \sv M \leq \stisv M = 0,
  \]
  we also obtain~$\sv M = 0$.
  
  In addition, if $\bigcap_{j \in \N} \Gamma_j = \{1\}$, then 
  the action~$\beta$ of~$\pi_1(M)$ on the corresponding coset tree is
  a free standard $\pi_1(M)$-space and~\cite[Theorem~2.6]{flps}
  \[ \lim_{j \rightarrow \infty} \frac{\isv {M_j}}{[\pi_1(M):\Gamma_j]}
     = \ifsv M^\beta.
  \]
  Moreover, $\ifsv M ^\beta = 0$ by Theorem~\ref{thm:mainergodic}.   
\end{proof}


\medskip
\vfill

\noindent
\emph{Daniel Fauser}\\
\emph{Stefan Friedl}\\
\emph{Clara L\"oh}\\[.5em]
  {\small
  \begin{tabular}{@{\qquad}l}
    Fakult\"at f\"ur Mathematik, 
    Universit\"at Regensburg, 
    93040 Regensburg\\
    \textsf{daniel.fauser@mathematik.uni-r.de},  
    \textsf{http://www.mathematik.uni-r.de/fauser}\\
    \textsf{stefan.friedl@mathematik.uni-r.de},
    \textsf{http://www.mathematik.uni-r.de/friedl}\\
    \textsf{clara.loeh@mathematik.uni-r.de}, 
    \textsf{http://www.mathematik.uni-r.de/loeh}
  \end{tabular}}
\end{document}